\documentclass[graybox]{svmult}

\usepackage{type1cm}        
\usepackage{makeidx}         
\usepackage{graphicx}        
\usepackage{multicol}        
\usepackage[bottom]{footmisc}

\usepackage{newtxtext}       %
\usepackage{newtxmath}       

\makeindex             

\newcommand{\cE}{{\mathcal E}}
\newcommand{\cF}{{\mathcal F}}

\newcommand{\cH}{{\mathcal H}}
\newcommand{\cI}{{\mathcal I}}
\newcommand{\cJ}{{\mathcal J}}
\newcommand{\cK}{{\mathcal K}}
\newcommand{\cL}{{\mathcal L}}

\newcommand{\cP}{{\mathcal P}}

\newcommand{\cS}{{\mathcal S}}

\newcommand{\cX}{{\mathcal X}}

\newcommand{\Om}{{\Omega}}
\newcommand{\om}{{\omega}}
\newcommand{\ve}{{\varepsilon}}
\newcommand{\del}{{\delta}}

\newcommand{\gam}{{\gamma}}

\newcommand{\Sig}{{\Sigma}}
\newcommand{\sig}{{\sigma}}
\newcommand{\al}{{\alpha}}
\newcommand{\be}{{\beta}}

\newcommand{\la}{{\lambda}}

\newcommand{\bbC}{{\mathbb C}}
\newcommand{\bbE}{{\mathbb E}}

\newcommand{\bbN}{{\mathbb N}}

\newcommand{\bbR}{{\mathbb R}}

\renewcommand{\epsilon}{\varepsilon}

\def\N{\mathbb{N}}

\begin{document}

\title*{Almost Sure Invariance Principle for Random Distance Expanding  Maps with a Nonuniform Decay of Correlations}
\titlerunning{Almost Sure Invariance Principle for Random Distance Expanding Maps}
\author{Davor Dragi\v cevi\' c and Yeor Hafouta}
\institute{Davor Dragi\v cevi\' c, Department of Mathematics, University of Rijeka, Rijeka, Croatia,  \email{ddragicevic@math.uniri.hr}
\and Yeor Hafouta, Department of Mathematics, The Ohio State University, Columbus OH USA,  \email{yeor.hafouta@mail.huji.ac.il}}
%
%
\maketitle

\abstract*{Each chapter should be preceded by an abstract (no more than 200 words) that summarizes the content. The abstract will appear \textit{online} at \url{www.SpringerLink.com} and be available with unrestricted access. This allows unregistered users to read the abstract as a teaser for the complete chapter.
Please use the 'starred' version of the \texttt{abstract} command for typesetting the text of the online abstracts (cf. source file of this chapter template \texttt{abstract}) and include them with the source files of your manuscript. Use the plain \texttt{abstract} command if the abstract is also to appear in the printed version of the book.}

\abstract{ We prove a quenched  almost sure invariance principle for certain classes of random distance expanding dynamical systems which do not necessarily exhibit  uniform decay of correlations.}

\section{Introduction}
\label{sec:1}

The aim of this note is to establish an almost sure invariance principle (ASIP) for certain classes of random dynamical systems.  More precisely, similarly to the setting introduced in~\cite{MSU}, 
the dynamics is formed by compositions \[ f_\omega^n:=f_{\sigma^{n-1}\omega}\circ \ldots \circ f_{\sigma\omega}\circ f_\omega,\,\om\in\Om \] of locally distance expanding maps  $f_\omega$ satisfying certain topological assumptions which are 
driven by an invertible, measure preserving transformation $\sigma$ on some probability space $(\Omega, \mathcal F, \mathbb P)$. Then, under suitable assumptions and for H\"older continuous observables 
 $\psi_\omega:X\to \mathbb{R}$, $\omega\in\Omega$ we establish a  quenched ASIP. Namely, we prove that for $\mathbb P$-a.e. $\omega \in \Omega$, the 
  random Birkhoff sums $\sum_{j=0}^{n-1}\psi_{\sig^j\om}\circ f_\om^j$ can be approximated in the strong sense by a sum of Gaussian independent random variables $\sum_{j=0}^{n-1}Z_j$ with the error being negligible compared to $n^{\frac12}$.  In comparison with the  previous results dealing with the ASIP for random or sequential dynamical systems,
the main novelty of our work  is that we do not require that our dynamics exhibits uniform (with respect to $\omega$) decay of correlations.

In a more  general setting and under suitable assumptions, Kifer proved in \cite{kifer} a central limit theorem (CLT)  and a law of iterated logarithm (LIL). As Kifer remarks, his arguments  (see \cite[Remark 4.1]{kifer}) also yield an ASIP  when there is an underlying random family of $\sig$-algebras which are sufficiently fast well mixing in an appropriate (random) sense (i.e. in the setup of~\cite[Theorem 2.1]{kifer}). In the context of random dynamics, Kifer's results can be applied  to random  expanding maps which admit a (random) symbolic representation. One of the main ingredients  in \cite{kifer} is a certain inducing argument, an approach that we also follow in the present paper. The main idea is that an ASIP for the original system will follow from an ASIP for a  suitably constructed induced system. 

For some classical work devoted to ASIP, we refer to~\cite{BP, PS}. In addition, 
we stress that there are quite a few works whose aim is to establish ASIP for deterministic dynamical systems. In this direction, we refer to the works of
 Field, Melbourne and T\"or\"ok \cite{FieldMelbourneTorok}, Melbourne and Nicol \cite{MN1, MN2}, and more  recently to Korepanov \cite{KO2, Korepanov}. In~\cite{GO},  Gou\"ezel  developed a new spectral technique for establishing ASIP, which was applied to certain classes of deterministic dynamical systems with the property that the corresponding transfer operator exhibits a spectral gap. 

Gou\"ezel's method was also used in \cite{ANV} to obtain  the annealed  ASIP for  certain classes of piecewise expanding  random dynamical systems.
In \cite{DFGTV1}  the authors proved for the first time (we recall that Kifer in~\cite{kifer} only briefly  commented that his methods also yield an ASIP) a quenched  ASIP for piecewise expanding  random dynamical systems, by invoking a recent ASIP for (reverse) martingales  due to Cuny and Merlevede \cite{CM} (which was also applied in many other deterministic and sequential setups; see for example~\cite{HNTV}). 
While the  type of maps $f_\om$ considered in  \cite{DFGTV1}  is more general than the ones considered in the present paper, in contrast to \cite{DFGTV1} in the present paper  we do  not assume a  uniform decay of correlations.  Moreover, the methods used in this paper can be extended to vector-valued observables $\psi_\om$ (see Remark~\ref{Remark}). On the other hand, it 
is unclear if the techniques in~\cite{DFGTV1} can be extended to the vector-valued case since the results in~\cite{CM} deal exclusively with the scalar-valued observables.  Finally, we mention our previous work~\cite{DH}, where we have obtained a quenched ASIP for certain classes of hyperbolic random dynamical systems. In addition, we 
have improved the main result from~\cite{DFGTV1}. However, the classes of dynamics we have considered again exhibit uniform decay of correlations.

Our techniques for establishing ASIP (besides the already mentioned inducing arguments), rely on a certain adaptation of 
the method of Gou\"ezel \cite{GO} which is of independent interest. Indeed, we first need to   modify Gou\"ezel's arguments and show that they yield an ASIP for  non-stationary sequences of random variables, which are not necessarily bounded in some $L^p$ space. 

We stress that our error term in ASIP  is of order $n^{1/4+O(1/p)}$ , where $p$ comes from certain $L^p$-regularity conditions we impose for the induced system. This is rather close to the  $n^{1/4}$ rate for deterministic uniformly expanding systems \cite{GO},  when $p\to\infty$ (although this rate was 
significantly improved by Korepanov~\cite{Korepanov}).

\section{Random distance expanding maps}

Let $(\Om,\cF, \mathbb P)$ be a complete probability space. Furthermore, let 
 $\sig:\Om\to\Om$ be an invertible   $\mathbb P$-preserving transformation  such that $(\Om,\cF, \mathbb P,\sig)$  is ergodic.  Moreover, 
let  $(\cX,\rho)$ be a compact metric space  normalized in size so that 
$\text{diam}\cX\leq 1$ together with the Borel $\sigma$-algebra $\mathcal B$, and
let  $\cE\subset\Om\times \cX$ be a  measurable set (with respect to the product $\sig$-algebra $\cF\times\mathcal B$) such that the fibers \[ \cE_\om=\{x\in \cX:\,(\om,x)\in\cE\},\quad\om\in\Om \] are compact. Hence (see \cite[Chapter III]{CV}),  it follows  that the map $\om\to\cE_\om$ is measurable with respect to the Borel $\sig$-algebra induced by the Hausdorff topology on the space $\cK(\cX)$ of compact subspaces of $\cX$. Moreover, the map $\omega \mapsto \rho(x,\cE_\om)$ is measurable for each $x\in \cX$.  
Finally, the projection map $\pi_\Om(\om,x)=\om$ is
measurable and it maps any $\cF\times\mathcal B$-measurable set to an
$\cF$-measurable set (see   \cite[Theorem III.23]{CV}). 

Let $f_\om \colon \cE_\om \to \cE_{\sig \om}$, $\omega \in \Omega$  be a family of surjective maps such that 
the map $(\om,x)\to f_\om (x)$ is measurable with respect to the $\sigma$-algebra $\cP$
which is the restriction of $\cF\times\mathcal B$ on $\cE$.
Consider the 
skew product transformation $F:\cE\to\cE$ given by 
\begin{equation}\label{Skew product}
F(\om,x)=(\sig\om,f_\om (x)).
\end{equation}
For $\omega \in \Omega$ and $n\in \N$, set
\[
f_\omega^n:=f_{\sigma^{n-1} \omega} \circ \ldots \circ f_\omega \colon \cE_\omega \to \cE_{\sigma^n \omega}.
\]
Let us now introduce several additional assumptions for the family $f_\omega$, $\omega \in \Omega$. More precisely, we require that:
\begin{itemize}
\item (\emph{topological exactness}) there exist a constant $\xi>0$ and a random variable $\omega \mapsto n_\om\in\bbN$ such that for
$\mathbb P$-a.e. $\omega \in \Omega$  and any $x\in \cE_\om$ we have that 
\begin{equation}\label{TopExRand}
f_\om^{n_\om}(B_\om(x,\xi))=
\cE_{\sig^{n_\om}\om},
\end{equation}
where 
$\,B_\om(x,r)$ denotes an open  ball in $\cE_\om$ centered in  $x$ with radius $r$;
\item (\emph{pairing property}) there exist random variables $\om \mapsto \gam_\om>1$ and $\om \mapsto D_\om\in\bbN$ such that for $\mathbb P$-a.e. $\omega \in \Omega$ and 
for any $x,x'\in\cE_{\sig\om}$ with $\rho(x,x')<\xi$ ($\xi$  comes from the previous assumption), 
we have that 
\begin{equation}\label{Pair1.0}
f_\om^{-1}(\{x\})=\{y_1,\ldots,y_k\}, \quad f_\om^{-1}(\{x'\})=\{y_1',\ldots,y_k'\},
\end{equation}
\[
k=k_{\om,x}=|f_\om^{-1}(\{x\})|\leq D_\om 
\]
and
\begin{equation}\label{Pair2.0}
\rho(y_i,y_i')\leq (\gam_\om)^{-1}\rho(x,x'), \quad \text{for $1\leq i\leq k$.}
\end{equation}
\end{itemize}
The above assumptions were considered in \cite{HK}, and they hold true in the setup of distance expanding maps considered in \cite{MSU}. We note that all the results stated in \cite{MSU} hold true under these assumptions (see \cite[Chapter 7]{MSU}) and not only under the assumptions from~\cite[Section 2]{MSU}.
For $\omega \in \Omega$ and  $n\in\bbN$,  set
\begin{equation}\label{gam om n D om n}
\gam_{\om,n}:=\prod_{i=0}^{n-1}\gam_{\sig^i\om}\,\,\text{ and }\,\,
D_{\om,n}:=\prod_{i=0}^{n-1}D_{\sig^i\om}.
\end{equation}
By induction, it follows from the pairing property  that for $\mathbb P$-a.e.  $\om \in \Omega$ and for any
 $x,x'\in\cE_{\sig^n\om}$ with $\rho(x,x')<\xi$, we have that 
\begin{equation}\label{Pair1}
(f_\om^n)^{-1}(\{x\})=\{y_1,\ldots ,y_k\}\,\,\text{ and }
\,\,(f_\om^n)^{-1}(\{x'\})=\{y_1',\ldots,y_k'\},
\end{equation}
where 
\[
k=k_{\om,x,n}=|(f_\om^n)^{-1}(\{x\})|\leq D_{\om,n},
\]
and 
\begin{equation}\label{Pair2}
\rho\big(f_\om^jy_i,f_\om^jy_i'\big)
\leq (\gam_{\sig^j\om,n-j})^{-1}\rho(x,x'), \ \text{for  $1\leq i\leq k$ and $0\leq j<n$.}
\end{equation}

Let $g:\cE\to\bbC$ be a measurable function. For any $\om\in\Om$,  consider the function $g_\om:=g(\om, \cdot)\colon \cE_\om \to \bbC$.
 For any 
$0<\al\leq 1$, set    
\[
v_{\al,\xi}(g_\om):=\inf\{R>0: |g_\om(x)-g_\om(x')|\leq R\rho^\al(x,x')\,\text{ if }\,
\rho(x,x')<\xi\},
\]
and let
\[
\|g_\om\|_{\al,\xi}=\|g_\om\|_\infty+v_{\al,\xi}(g_\om),
\]
where $\|\cdot\|_\infty$ denotes the supremum norm and $\rho^\al(x,x'):=\big(\rho(x,x')\big)^\al$. 
We emphasize that these norms are  $\cF$-measurable (see~\cite[p. 199]{HK}).

Let $\cH_\om^{\al,\xi}=(\cH_\om^{\al,\xi},\|\cdot\|_{\al,\xi})$ denote the space of all $h:\cE_\om\to\bbC$ such that $\|h\|_{\al,\xi}<\infty$.  Moreover, let $\cH_{\om,\bbR}^{\al,\xi}$ be  the space of all real-valued functions
in $\cH_\om^{\al,\xi}$. 

Take a random variable  $H\colon \Omega \to [1, \infty)$ such that \[\int_{\Omega}\ln H_\om \, d\mathbb P(\om)<\infty,\] where $H_\om:=H(\om)$. Moreover, 
 let $\cH^{\al,\xi}(H)$ be the set of all measurable functions $g:\cE\to\bbC$ satisfying
$v_{\al,\xi}(g_\om)\leq H_\om$ for $\om \in \Omega$. Furthermore, for $\om \in \Omega$ set
\[
\cH_\om^{\al,\xi}(H):=\{g\colon \cE_\om \to \mathbb C:  \text{$g$ measurable and $v_{\al,\xi}(g)\leq H_\om $}\}
\]
and
\begin{equation}\label{Q H def 1}
Q_\om(H)=\sum_{j=1}^\infty H_{\sig^{-j}\om}(\gam_{\sig^{-j}\om,j})^{-\al}.
\end{equation}
Since $\om \mapsto \ln H_\omega$ is integrable, we have (see \cite[Chapter 2]{MSU}) that $Q_\om(H)<\infty$ for $\mathbb P$-a.e. $\om \in \Omega$. The following simple distortion property is a direct consequence of
(\ref{Pair2}).
\begin{lemma}\label{Distor Lemm}
Take $\om \in \Omega$, $n\in \mathbb N$ and   $\varphi =(\varphi_0, \ldots, \varphi_{n-1})$, where $\varphi_i \in \cH_{\sigma^i \om}^{\al,\xi}(H)$ for $0\le i \le n-1$.
 Set
\[
\cS_n^\om \varphi:=\sum_{j=0}^{n-1}\varphi_j \circ f_\om^j.
\]
Furthermore, take  $x,x'\in\cE_{\sig^n\om}$ such that  $\rho(x,x')<\xi$ and let $y_i, y_i'$, $1\le i\le k$ be as in~(\ref{Pair1}). Then, 
for any $1\leq i\leq k$ we have that 
\begin{eqnarray}\label{S n distortion}
|\cS_n^\om \varphi(y_i)-\cS_n^\om \varphi(y_i')|
\leq\rho^\al(x,x') 
Q_{\sig^n\om}(H).\nonumber
\end{eqnarray}
\end{lemma}

\subsection{Transfer operators}
Let us take an observable $\psi \colon \cE \to \mathbb R$ such that $\psi \in  \cH^{\al,\xi}(H)$. We consider the associated random Birkhoff sums
\[
S_n^\om \psi=\sum_{i=0}^{n-1}\psi_{\sig^i\om}\circ f_\om^i, \quad \text{for $n\in \mathbb N$ and $\omega \in \Omega$.}
\]
Furthermore, suppose that $\phi \colon \cE \to \mathbb R$ also belongs to $\cH^{\al,\xi}(H)$. For $\om \in \Omega$, $z\in \mathbb C$ and $g \colon \cE_\omega \to \mathbb C$, we define
\begin{eqnarray}\label{TraOp}
\cL_\om^z g(x)=\sum_{y\in f_\om^{-1}(\{x\})}e^{\phi_\om(y)+z\psi_\om(y)}g(y).
\end{eqnarray}
It follows from~\cite[Theorem 5.4.1.]{HK} that $\cL_\om^z \colon \cH_\om^{\al,\xi} \to \cH_{\sig \om}^{\al,\xi}$ is a well-defined and bounded linear operator for each $\om \in \Omega$ and $z\in \mathbb C$. Moreover, the map $z\mapsto \cL_\om^z$ is analytic for each $\om \in \Omega$.

Let us denote $\cL_\om^0$ simply by $\cL_\om$. It follows from~\cite[Theorem 3.1.]{MSU} that for $\mathbb P$-a.e. $\omega \in \Omega$,  there exists a triplet $(\la_\om,h_\om,\nu_\om)$ consisting of a positive number $\la_\om>0$, a strictly positive function 
$h_\om\in\cH_\om^{\al,\xi}$ and a probability measure $\nu_\om$ on $\cE_\om$ so that 
\[
\cL_\om h_\om=\la_\om h_{\sig\om},\,(\cL_\om)^*\nu_{\sig\om}=\la_\om \nu_{\om}, \,\nu_\om(h_\om)=1, 
\]
and that maps $\om \mapsto \la_\om$, $\om \mapsto h_\om$ and $\om \mapsto \nu_\om$ are measurable. We can assume without any loss of generality that $\la_\om=1$ for $\mathbb P$-a.e. $\om \in \Omega$ (since otherwise we can replace $\cL_\om$ with $\cL_\om /\la_\om$). For $\mathbb P$-a.e. $\omega \in \Omega$, let $\mu_\om$ be a measure on $\cE_\om$ 
given by $d\mu_\om:=h_\om d\nu_\om$. We recall (see~\cite[Lemma 3.9]{MSU}) that these measures satisfy the so-called  {\em equivariant property}, i.e. we have that \begin{equation}\label{8:11} f^*_{\omega}\mu_{\omega}=\mu_{\sigma \omega}, \quad \text{for $\mathbb P$-a.e. $\om \in \Omega$.} \end{equation} Moreover, these measures give rise to a measure $\mu$ on $\Omega \times \cE$ with the property that 
for any   $A \in \mathcal F \times  \mathcal B$,  \[\mu(A)=\int_\Omega \mu_{\omega}(A_{\omega}) d\mathbb{P}(\omega),\] where $A_{\omega}=\{x\in \cE_\om; (\omega, x)\in A\}$. Then,  $\mu$ is invariant for the skew-product transformation $F$ given by~\eqref{Skew product}.  Moreover, $\mu$ is ergodic. 

For $\overline{t}=(t_0,\ldots ,t_{n-1})\in \mathbb R^n$, set 
\[
\cL_{\om}^{\overline{t},n}:=\cL^{it_{n-1}}_{\sig^{n-1}\om}\circ\ldots\circ\cL^{it_1}_{\sig\om}\circ\cL^{it_0}_{\om}.
\]
Moreover, let $\cL_\om^{n}:=\cL_\om^{\overline{0}, n}$, where $\overline{0}=(0, \ldots, 0)\in \mathbb R^n$. Note that  
\[
\|\cL_\om^{n}\textbf{1}\|_\infty\leq (\deg f_\om^n)\cdot  e^{\|S_n^\om \phi\|_\infty}
\leq D_{\om,n} e^{\|S_n^\om \phi\|_\infty}<\infty,
\]
where $\textbf{1}$ is the function taking constant value $1$ and
\[
\deg f_\omega^n:=\sup_{x\in \mathcal E_{\sigma^n \omega}} \lvert (f_\omega^n)^{-1}(\{x\}) \rvert.
\]
\begin{lemma}\label{L-Y-general}
For any $\mathbb P$-a.e. $\om \in \Om$ we have that for any $n\in\bbN$,  $T>0$, $ \overline{t}=(t_0,\ldots,t_{n-1})\in[-T,T]^n$
and $g\in \cH_\om^{\al,\xi}$,
\begin{equation*}
v_{\al,\xi}(\cL_\om^{\overline t,n}g)\leq \|\cL_\om^{n}\textbf{1}\|_\infty
\big(v_{\al,\xi}(g)(\gam_{\om,n})^{-\al}
+2Q_{\sig^n\om}(H)(1+T)\|g\|_\infty\big).
\end{equation*} 
Consequently, 
\begin{equation}\label{L.Y.-general}
\|\cL_\om^{\overline t,n}g\|_{\al,\xi}\leq \|\cL_\om^{n}\textbf{1}\|_\infty
\big(v_{\al,\xi}(g)(\gam_{\om,n})^{-\al}
+(1+2Q_{\sig^n\om}(H))(1+T)\|g\|_\infty\big).
\end{equation}
\end{lemma}

\begin{proof}
The proof is similar to the proof of~\cite[Lemma 5.6.1.]{HK}, but for reader's convenience all the details are given. The idea is to apply  Lemma~\ref{Distor Lemm} for $\varphi=(\varphi_0, \ldots, \varphi_{n-1})$ given by
\[ \varphi_j:=\phi_{\sigma^j \om}+it_j \psi_{\sigma^j \om}, \quad \text{for $0\le j \le n-1$.} \]
Set $A_n^\om=\sum_{j=0}^{n-1}t_j \psi_{\sigma^j \om}\circ f_\om^j$.
Firstly, by the definition of $\cL_\om^n$ we have
\begin{equation}\label{9:23}
\|\cL_\om^{\overline t,n}g\|_\infty\leq\|g\|_\infty\|\cL_\om^{n}\textbf1\|_\infty.
\end{equation}
In order to complete the proof of the lemma we need to approximate $v_{\al,\xi}(\cL_\om^{\overline t,n}g)$.  Let
$x,x'\in\cE_{\sigma^n\om}$ be such that
$\rho(x,x')<\xi$ and let $y_1,\ldots,y_k$ and $y_1',\ldots,y_k'$ be the points in
$\cE_\om$ satisfying (\ref{Pair1.0}) and (\ref{Pair2.0}). We can write
\[
\begin{split}
& \big|\cL_\om^{\overline t,n}g(x)-\cL_\om^{\overline t,n}g(x')\big|\\
&=\big|\sum_{q=1}^k\big(e^{S_n^\om \phi(y_q)+iA_n^\om(y_q)}g(y_q)-
e^{S_n^\om \phi(y_q')+iA_n^\om(y_q')}g(y_q')\big)\big|
\\
&\leq\sum_{q=1}^ke^{S_n^\om \phi(y_q)}
|e^{iA_n^\om(y_q)}g(y_q)-e^{iA_n^\om (y_q')}g(y_q')|\\
&\phantom{=}+
\sum_{q=1}^k|e^{iA_n^\om(y_q')}g(y_q')| \cdot 
|e^{S_n^\om \phi(y_q)}-e^{S_n^\om \phi(y_q')}| 
=:
I_1+I_2.
\end{split}
\]
In order to estimate $I_1$, observe that for any $1\leq q\leq k$,
\[
\begin{split}
&|e^{iA_n^\om (y_q)}g(y_q)-e^{iA_n^\om (y_q')}g(y_q')|\\
&\leq
|g(y_q)|\cdot|e^{iA_n^\om (y_q)}-e^{iA_n^\om (y_q')}|+
|g(y_q)-g(y_q')|
=:J_1+J_2.
\end{split}
\]
By the mean value theorem and then by Lemma~\ref{Distor Lemm},
\[
J_1\leq2T\|g\|_\infty  Q_{\sigma^n\om}(H)\rho^\al(x,x'),
\]
while by (\ref{Pair2}),
\[
J_2\leq v_{\al,\xi}(g)\rho^\al(y_q,y'_q)\leq
 v_{\al,\xi}(g)(\gam_{\om,n})^{-\al}\rho^\al(x,x').
\] 
It follows that
\begin{equation*}
I_1\leq \cL_{\om}^{n}\textbf1(x)
\big(2T\|g\|_\infty Q_{\sigma^n\om}(H)+v_{\al,\xi}(g)
(\gam_{\om,n})^{-\al}\big)\rho^\al(x,x').
\end{equation*}
Next, we  estimate $I_2$.  By the mean value theorem
and Lemma \ref{Distor Lemm},
\begin{eqnarray*}
|e^{S_n^\om \phi(y_q)}-
e^{S_n^\om \phi(y'_q)}|\leq Q_{\sigma^n\om}(H)\cdot
\max\{e^{S_n^\om \phi(y_q)},
e^{S_n^\om \phi(y'_q)}\}
\rho^\al(x,x')
\end{eqnarray*}
 and therefore
\[
\begin{split}
I_2 & \leq\|g\|_\infty(\cL_{\om}^{n}\textbf1(x)+\cL_{\om}^{n}\textbf1(x'))
Q_{\sigma^n\om}(H)\rho^\al(x,x')\\
&\leq
2\|g\|_\infty\|\cL_\om^n \textbf1\|_\infty     
Q_{\sigma^n\om}(H)\rho^\al(x,x'),                           
\end{split}
\]
yielding the first
statement of the lemma and~\eqref{L.Y.-general} follows from~\eqref{9:23}, together with the first statement.
\end{proof}

By Lemma~\ref{L-Y-general} together with the observation that $(\gam_{\om,n})^{-\al}\le 1$, we conclude   that  there exists a random variable $C \colon \Omega \to [1, \infty)$ such that for $\mathbb P$-a.e. $\omega \in \Omega$, $n\in \N$ and  for any $\overline{t}=(t_0,t_1,\ldots,t_{n-1})\in[-1,1]^n$, we have that 
\begin{equation}\label{552}
\lVert  \cL_\om^{\overline t,n} \rVert_{\al,\xi}  \leq  C(\sig^n \om) \| \cL_\om^{n}\textbf{1}\|_\infty, 
\end{equation}
where $\lVert  \cL_\om^{\overline t,n} \rVert_{\al,\xi}$ denotes  the operator norm of $\cL_\om^{\overline t,n}$  when considered as a linear operator from $\cH_{\omega}^{\alpha,\xi}$ to 
 $\cH_{\sigma^n\omega}^{\alpha,\xi}$.
Note that we can just take $C(\om)=4(1+Q_\omega)$.
For $\mathbb P$-a.e. $\omega \in \Omega$, we define $\hat{\cL}_\om \colon \cH_\om^{\al,\xi} \to \cH_{\sigma \om}^{\al,\xi}$ by
\[
\hat{\cL}_\om g=\cL_\om (gh_\om)/h_{\sigma \om}, \quad g\in  \cH_\om^{\al,\xi}.
\]
Moreover, for $n\in \mathbb N$, set \[\hat{\cL}_\om^{n}:=\hat{\cL}_{\sigma^{n-1} \om}\circ \ldots \circ  \hat{\cL}_{\sigma \om}\circ  \hat{\cL}_\om.\] Clearly, 
\[
\hat{\cL}_\om^{n}g=\cL_\om^n(gh_\om)/h_{\sigma^n \om}, \quad \text{for $g\in \cH_\om^{\al,\xi}$ and $n\in \mathbb N$.}
\]
We need the following result which is a direct consequence of~\cite[Lemma 3.18.]{MSU}.
\begin{lemma}\label{ds}
There exist $\lambda >0$ and a random variable $K\colon \Omega \to (0, \infty)$ such that 
\[
\lVert \hat{\cL}_\om^n g\rVert_\infty \leq \max(1,1/Q_\om) K(\sigma^n \omega)e^{-\lambda n} \lVert g\rVert_{\al,\xi}, 
\]
for $\mathbb P$-a.e. $\om \in \Omega$, $n\in \mathbb N$ and $g\in \cH_\om^{\al,\xi}$ such that $\int_{\cE_\om}g\, d\mu_\om=0$.
\end{lemma}
Applying Lemma~\ref{ds} with the function $g=1/h_\omega -1$, and taking into account that $\cL_\om^n h_\om=h_{\sigma^n\om}$ (since $\la_\om=1$),
it follows from~\eqref{552}   that  for $\mathbb P$-a.e. $\omega \in \Omega$, $n\in \N$ and  for any $\overline{t}=(t_0,t_1,\ldots,t_{n-1})\in[-1,1]^n$, 
\begin{equation}\label{552n}
\lVert  \cL_\om^{\overline t,n} \rVert_{\al,\xi}  \leq  (1+U(\omega))K(\sig^n\om) C'(\sig^n \om)
\end{equation}
where $C'(\omega)=C(\omega)\|h_\om\|_\infty$ and $U(\omega)=\max(1,1/Q_\om)\cdot(1+\|1/h_\om\|_{\al,\xi})$.

\section{A refined version of Gou\"ezel's theorem}
In this section we present a  more general version of Gou\"ezel's almost sure invariance principle for non-stationary processes~\cite[Theorem 1.3.]{GO}. This result will than be used in the next section to obtain the almost sure invariance principle for random distance expanding maps. 

Let $(A_1, A_2, \ldots )$ be an $\mathbb R$-valued process on some probability space $(\Omega, \mathcal F, \mathbb P)$. 
We first recall the  condition that we denote (following~\cite{GO}) by (H): there exist $\ve_0>0$ and $C,c>0$ such that for any $n,m>0$, $b_1<b_2< \ldots <b_{n+m+k}$, $k>0$ and $t_1,\ldots, t_{n+m}\in\bbR$ with $|t_j|\leq\ve_0$, we have that
\begin{eqnarray*}
\Big|\bbE\big(e^{i\sum_{j=1}^nt_j(\sum_{\ell=b_j}^{b_{j+1}-1}A_\ell)+i\sum_{j=n+1}^{n+m}t_j(\sum_{\ell=b_j+k}^{b_{j+1}+k-1}A_\ell)}\big)\\
-\bbE\big(e^{i\sum_{j=1}^nt_j(\sum_{\ell=b_j}^{b_{j+1}-1}A_\ell)}\big)\cdot\bbE\big(e^{i\sum_{j=n+1}^{n+m}t_j(\sum_{\ell=b_j+k}^{b_{j+1}+k-1}A_\ell)}\big)\Big|\\\leq C(1+\max|b_{j+1}-b_j|)^{C(n+m)}e^{-ck}.
\end{eqnarray*}

\begin{theorem}\label{Gouzel Thm}
Suppose that $(A_1,A_2,\ldots)$ is an  $\mathbb R$-valued centered process on the probability space $(\Omega, \mathcal F, \mathbb P)$ that
 satisfies~(H). Furthermore, assume that:
\begin{itemize}
\item there exist  $u>0$ and $L\in\bbN$ such  that for any $n, m\in \bbN$, $m\geq L$ we have that
\begin{equation}\label{Go1}
Var \bigg{(}\sum_{j=n+1}^{n+m}A_j \bigg{)}\geq um;
\end{equation}

\item there exist constants $p\geq 6$ and $a,C>0$ such  that for any $n\in \bbN$ we have
\begin{equation}\label{Go2}
\|A_n\|_{L^p}\leq a n^{\frac 1p}.
\end{equation}
In addition, for any  $n, m\in \N$ the finite sequence $(A_i/(n+m)^{1/p})_{n+1\le i \le n+m}$ also satisfies condition~(H) with the same constants $\ve_0$, $C$ and $c$.
\end{itemize}
Then for any $\delta>0$,   there exists a coupling between $(A_j)$ and  a  sequence $(B_j)$  of independent centered normal random variables such that 
\begin{equation}\label{8:44}
\left|\sum_{j=1}^{n}(A_j-B_j)\right|=o(n^{a_p+\delta}) \quad a.s., 
\end{equation}
where \[ a_p=\frac{p}{4(p-1)}+\frac 1p.\] Moreover, there exists a constant $C>0$ such that for any $n\in \N$,
\begin{equation}\label{Var est}
\Big\|\sum_{j=1}^n A_j\Big\|_{L^2}-Cn^{a_p+\delta}\leq \Big\|\sum_{j=1}^n B_j\Big\|_{L^2}\leq \Big\|\sum_{j=1}^n A_j\Big\|_{L^2}+Cn^{a_p+\delta}.
\end{equation} 
Finally, there exists a coupling between  $(A_j)$ and a standard Brownian motion $(W_t)_{t\ge 0}$ such that 
\[
\left|\sum_{j=1}^{n}A_j-W_{\sigma_n^2}\right|=o(n^{\frac12 a_p+\frac14+\delta}) \quad a.s., 
\]
where
\[
\sigma_n=\Big\|\sum_{j=1}^n A_j\Big\|_{L^2}.
\]
\end{theorem}

\begin{remark}\label{Remark}
The above result (together with its proof) is similar to~\cite[Theorem 1.3]{GO}. However, we stress that~\cite[Theorem 1.3]{GO} requires that the process $(A_1, A_2, \ldots)$ is bounded in $L^p$, while the above Theorem~\ref{Gouzel Thm} works under the assumption that~\eqref{Go2} holds. Consequently, the estimate for the error term in~\eqref{8:44}
is different from that in~\cite[Theorem 1.3]{GO}.

Note also that our condition  \eqref{Go1} replaces condition (1.3) in~\cite[Theorem 1.3]{GO}. This, of course, makes it impossible to get a precise formula for the variance of the approximating Gaussian random variables $\sum_{j=1}^{n}B_j$, as in \cite{GO}. 
However,  in our context we have the estimate~(\ref{Var est}). Observe that (\ref{Var est}) together with (\ref{Go1}) ensures that 
\[
\lim_{n\to \infty} \frac{\Big\|\sum_{j=1}^n B_j\Big\|_{L^2}}{\Big\|\sum_{j=1}^n A_j\Big\|_{L^2} }=1.
\]
 Therefore, Theorem \ref{Gouzel Thm} yields a corresponding almost sure version of the CLT for the sequence $\frac{1}{a_n}\sum_{j=1}^n A_j$, where $a_n=\|\sum_{j=1}^n A_j\|_{L^2}$. As we have mentioned,  a precise formula for the variance of the approximating Gaussian random variables in the context of~\cite[Theorem 1.3]{GO}
 was obtained in~\cite[Lemma 5.7]{GO}. Hence,  in our modification of the proof of ~\cite[Theorem 1.3]{GO} we will not need an appropriate version of \cite[Lemma 5.7]{GO} (and instead we will prove (\ref{Var est}) directly).

We  also note  that our modification of the arguments in \cite{GO} also yields  a certain convergence rate for $p\in(4,6)$, but in order to keep our exposition as simple as possible we have formulated the results only under the assumption that $p\geq 6$.

Finally, we remark that like in~\cite{GO} we can consider processes taking values in $\mathbb R^d$ and that Theorem~\ref{Gouzel Thm} holds in this case also. We prefer to work with processes in $\mathbb R$ to keep our exposition as simple as possible. 

\end{remark}

\begin{proof}[Proof of Theorem~\ref{Gouzel Thm}]
We follow step by step the proof of~\cite[Theorem 1.3]{GO} by making necessary adjustments.
  Firstly, applying ~\cite[Proposition 4.1]{GO} with the finite sequence $(A_i/(n+m)^{1/p})_{n+1\le i\le n+m}$, we get that for each $\eta>0$ there exists $C>0$ such that 
\begin{equation}\label{Modified bound}
\left\|\sum_{j=n+1}^{n+m}A_j\right\|_{L^{p-\eta}}\leq Cm^{\frac 12}(n+m)^{1/p}, \quad \text{for $m, n\ge 0$.}
\end{equation}
We note that although~\cite[Proposition 4.1]{GO} was  formulated for an  infinite sequence,  the proof for a  finite sequence proceeds by using the same arguments. 
We consider the so-called  big and small blocks as  introduced in~\cite[p.1659]{GO}. Fix $\beta\in(0,1)$ and $\ve\in(0,1-\beta)$. Furthermore, let  $f=f(n)=\lfloor \beta n \rfloor$. Then,  Gou\"ezel decomposes $[2^n,2^{n+1})$ into a union of $F=2^f$ intervals $(I_{n,j})_{0\leq j<F}$ of the same length, and $F$ gaps $(J_{n,j})_{0\leq j<F}$ between them. In other words, we have 
\[
[2^n, 2^{n+1})=J_{n,0}\cup I_{n, 0}\cup J_{n,1}\cup I_{n,1}\cup \ldots \cup J_{n, F-1}\cup I_{N, F-1}.
\]
Let us outline the construction of this decomposition. For $1\le j<F$, we write $j$ in the form $j=\sum_{k=0}^{f-1} \alpha_k(j)2^k$ with $\alpha_k \in \{0, 1\}$. We then take the smallest $r$ with the property that $\alpha_r(j)\neq 0$ and take $2^{\lfloor \epsilon n\rfloor}2^r$ to be the length of $J_{n,j}$. 
In addition, the length of $J_{n,0}$ is $2^{\lfloor \epsilon n\rfloor}2^f$. Finally, the length of each interval $I_{n,j}$ is $2^{n-f}-(f+2)2^{\lfloor \epsilon n\rfloor-1}$.

In addition,  we recall some notations from~\cite{GO} which we will also use.  We define a  partial order on  $\{(n,j):\,n\in\bbN,\,0\leq j<F(n)\}$ by writing  $(n,j)\prec (n',j')$ if the interval $I_{n,j}$ is to the left of $I_{n',j'}$. Observe that a sequence  $((n_k,j_k))_k$ tends to infinity if and only if $n_k\to\infty$. Moreover, let 
\[
X_{n,j}:=\sum_{\ell\in I_{n,j}}A_\ell
\]
and \[\cI:=\bigcup_{n,j}I_{n,j} \quad  \text{and} \quad \cJ:=\bigcup_{n,j}J_{n,j}.\] The rest of the proof will be divided (following again~\cite{GO}) into six steps.

\emph{First step:} We first prove the following version of~\cite[Proposition 5.1]{GO}.
\begin{proposition}\label{bb}
There exists a coupling between $(X_{n,j})$ and $(Y_{n,j})$ such that, almost surely, when $(n,j)$ tends to infinity, 
\[
\left|\sum_{(n',j')\prec (n,j)}X_{n',j'}-Y_{n',j'}\right|=o(2^{(\beta+\ve)n/2}).
\]
Here, $(Y_{n,j})$ is a family of independent random variables such that  $Y_{n,j}$ and  $X_{n,j}$ are equally distributed. 
\end{proposition}
Before we outline the proof of Proposition~\ref{bb}, we will first introduce some preparatory material. Let $\tilde{X}_{n,j}=X_{n,j}+V_{n,j}$, where the $V_{n,j}$'s are independent copies of the random variable $V$ constructed in~\cite[Proposition 3.8]{GO}, which are independent of everything else (enlarging  our probability space if necessary). 
Write $X_n=(X_{n,j})_{0\leq j< F(n)}$ and $\tilde X_n=(\tilde X_{n,j})_{0\leq j< F(n)}$. Then,  we have the following version of~\cite[Lemma 5.2]{GO}.
\begin{lemma}\label{Lemma 5.2}
Let $\tilde Q_n$ be a random variable distributed like $\tilde X_n$, but independent of $(\tilde X_1,\ldots,\tilde X_{n-1})$. We have
\begin{equation}\label{(5.3)}
\pi\big((\tilde X_1,\ldots ,\tilde X_{n-1},\tilde X_n),
(\tilde X_1,\ldots ,\tilde X_{n-1},\tilde Q_n)\big)\leq C4^{-n},
\end{equation}
where  $\pi(\cdot,\cdot)$ is the Prokhorov metric (see~\cite[Definition 3.3]{GO}) and $C>0$ is some constant not depending on $n$.
\end{lemma}
\begin{proof}[Proof of Lemma~\ref{Lemma 5.2}]
The proof is carried out by repeating the proof of~\cite[Lemma 5.2]{GO} with one slight modification. 
For reader's convenience we provide a complete proof.

The random process $(X_1,\ldots ,X_n)$ takes its values in $\mathbb R^D$, where $D=\sum_{m=1}^n F(m)\leq C2^{\be n}$. Moreover, each component in $\mathbb R$ of this process is one of the $X_{n,j}$, hence it is a sum of at most $2^n$ consecutive variables $A_\ell$. On  the other hand, the interval $J_{n,0}$ is a gap between $(X_j)_{j<n}$ and $X_n$, and its length $k$ is $C^{\pm 1}2^{\ve n+\be n}$. Let $\phi$ and $\gamma$ denote the respective characteristic functions of $(X_1,\ldots,X_{n-1},X_n)$ and $(X_1,\ldots ,X_{n-1},Q_n)$, where $Q_n$ is distributed like $X_n$ and is independent of $(X_1,\ldots ,X_{n-1})$. The assumption (H) ensures that for Fourier parameters $t_{m,j}$ all bounded by $\ve_0$, we have
\[
|\phi-\gamma|\leq C(1+2^n)^{CD}e^{-ck}\leq Ce^{-c'2^{\be n+\ve n}},
\]
if $n$ is large enough. Let $\tilde\phi$ and $\tilde\gamma$ be the characteristic functions of, respectively, $(\tilde X_1,\ldots,\tilde X_n)$ and $(\tilde X_1,\ldots,\tilde X_{n-1},\tilde Q_n)$: they are obtained by multiplying $\phi$ and $\gamma$ by the characteristic function of $V$ is each variable. Since this function is supported in $\{|t|\leq\ve_0\}$, we obtain, in particular, that 
\[
|\tilde\phi-\tilde\gamma|\leq Ce^{-c2^{\be n+\ve n}}.
\]

We then use~\cite[Lemma 3.5.]{GO} with $N=D$ and $T'=e^{2^{\ve n/2}}$ to obtain that
\[
\begin{split}
&\pi((\tilde X_1,\ldots ,\tilde X_n),(\tilde X_1,\ldots ,\tilde X_{n-1},\tilde Q_n))\\
&\leq \sum_{m\leq n}\,\sum_{j<F(m)}\mathbb P(|\tilde{X}_{m,j}|\geq e^{2^{\ve n/2}})+e^{CD2^{\ve n/2}}e^{-c2^{\be n+\ve n}}.
\end{split}
\]
So far our arguments were identical to those in the proof of~\cite[Lemma 5.2]{GO}. In the rest of the proof we will introduce the above mentioned modification of the arguments from~\cite{GO}.
Using the Markov inequality,  we obtain that
\[
\mathbb P(|\tilde X_{m,j}|\geq e^{2^{\ve n/2}})\leq 
e^{-2^{\ve n/2}}\bbE |\tilde X_{m,j}|.
\]
However, since $\|A_l\|_{L^p}\leq al^{1/p}$ for every $l\in \N$ (and for some constant $a>0$), we have that  $\bbE |\tilde X_{m,j}|\leq C2^{n+\frac np}$. Summing the resulting upper bounds for $\mathbb P(|\tilde X_{m,j}|\geq e^{2^{\ve n/2}})$,  we obtain the desired result.
\end{proof}

The following result follows from Lemma \ref{Lemma 5.2} exactly in the same way as~\cite[Corollary 5.3]{GO} follows from~\cite[Lemma 5.2]{GO}.

\begin{corollary}\label{Cor 5.3}
Let $\tilde R_n=(\tilde R_{n,j})_{j<F(n)}$ be distributed like $\tilde X_n$ and such that the $\tilde R_n$ are independent of each other. Then there exist $C>0$ and a coupling between $(\tilde X_1,\tilde X_2,\ldots)$ and $(\tilde R_1,\tilde R_2,\ldots)$ such that for all $(n,j)$,
\[
\mathbb P(|\tilde X_{n,j}-\tilde R_{n,j}|\geq C4^{-n})\leq C4^{-n}.
\]
\end{corollary}

We also need the following version of~\cite[Lemma 5.4]{GO}.
\begin{lemma}\label{Lem 5.4}
For any $n\in\bbN$, we have
\[
\pi\Big((\tilde R_{n,j})_{0\leq j<F(n)},(\tilde Y_{n,j})_{0\leq j<F(n)}\Big)\leq C4^{-n}
\]
where $\tilde Y_{n,j}=Y_{n,j}+V_{n,j}$.
\end{lemma}

\begin{proof}[Proof of Lemma~\ref{Lem 5.4}]
We follow the proof of~\cite[Lemma 5.4]{GO}. We define $\tilde Y_{n,j}^i$ for $0\le i \le f$ as follows: for $0\le k <2^{f-i}$, the random vector $\tilde{\mathcal Y}_{n,k}^i:=(\tilde Y_{n, j}^i)_{k2^i \le j< (k+1)2^i}$ is distributed as $(\tilde X_{n,j})_{k2^i \le j< (k+1)2^i}$, and 
$\tilde{\mathcal Y}_{n,k}^i$ is independent of $\tilde{\mathcal Y}_{n,k'}^i$ when $k\neq k'$. 
Set $\tilde Y^i=(\tilde Y_{n,j}^i)_{0\le j<F}$, for $0\le i \le f$. By~\cite[(5.7)]{GO}, we have that
\begin{equation}\label{aux1}
\pi (\tilde Y^i, \tilde Y^{i-1})\le \sum_{k=0}^{2^{f-i}-1}\pi (\tilde{\mathcal Y}_{n,k}^i, (\tilde{\mathcal Y}_{n, 2k}^{i-1}, \tilde{\mathcal Y}_{n,2k+1}^{i-1})),
\end{equation}
for $1\le i\le f$. As in the proof of~\cite[Lemma 5.4]{GO}, as a consequence of the condition (H), the difference between the characteristic functions of $\tilde{\mathcal Y}_{n,k}^i$ and $(\tilde{\mathcal Y}_{n, 2k}^{i-1}, \tilde{\mathcal Y}_{n,2k+1}^{i-1})$ is at most
$Ce^{-c'2^{\epsilon n+i}}$ for $n$ large enough. Hence, by applying~\cite[Lemma 3.5]{GO} with $N=2^i$ and $T'=e^{2^{\epsilon n/2}}$ we obtain that
\[
\begin{split}
& \pi (\tilde{\mathcal Y}_{n,k}^i, (\tilde{\mathcal Y}_{n, 2k}^{i-1}, \tilde{\mathcal Y}_{n,2k+1}^{i-1})) \\
& \le \sum_{j=k2^i}^{(k+1)2^i-1}\mathbb P(\lvert \tilde X_{n,j}\rvert \ge e^{2^{\epsilon n/2}}) +Ce^{2^{\epsilon n/2+i}}e^{-c'2^{\epsilon n+i}}.
\end{split}
\]
By estimating $\mathbb P(\lvert \tilde X_{n,j}\rvert \ge e^{2^{\epsilon n/2}})$ as in the proof of Lemma~\ref{Lemma 5.2},  we conclude that
\begin{equation}\label{aux2}
\pi (\tilde{\mathcal Y}_{n,k}^i, (\tilde{\mathcal Y}_{n, 2k}^{i-1}, \tilde{\mathcal Y}_{n,2k+1}^{i-1})) \le Ce^{-2^{\delta n}},
\end{equation}
for some $\delta >0$.  The conclusion of the lemma now follows from~\eqref{aux1} and~\eqref{aux2} by summing over $i$ and noting that the  process $(\tilde Y_{n,j}^f)_{0\le j<F}$ coincides with $(\tilde R_{n,j})_{0\le j<F}$ and that $(\tilde Y_{n,j}^0)_{0\le j<F}$ coincides with $(\tilde Y_{n,j})_{0\le j<F}$.
\end{proof}

Finally, relying on Corollary \ref{Cor 5.3} and  Lemma \ref{Lem 5.4} , the proof of Proposition \ref{bb} is completed exactly as in \cite{GO}.
\qed

\emph{Second step:} We now establish the version of~\cite[Lemma 5.6]{GO}. 
We first recall the following result (see~\cite[Corollary 3]{Zai} or~\cite[Proposition 5.5]{GO}).
\begin{proposition}\label{Cor3}
Let $Y_0,\ldots ,Y_{b-1}$ be independent centered $\mathbb R^d$-valued random vectors. Let $q\geq2$ and set $M=\big(\sum_{j=0}^{b-1}\mathbb E|Y_j|^q\big)^{1/q}$. Assume that there exists a sequence $0=m_0<m_1<\ldots<m_s=b$ such that with $\zeta_k=Y_{m_k}+\ldots+Y_{m_{k+1}-1}$ and $B_k=\text{Cov}(\zeta_k)$, for any $v\in\mathbb R^d$ and $0\leq k<s$ we have that
\begin{equation}\label{BLOCKS}
100M^2|v|^2\leq B_kv\cdot v\leq 100CM^2|v|^2,
\end{equation}
where $C\geq1$ is some constant. Then, there exists a coupling between $(Y_0,\ldots,Y_{b-1})$ and a sequence of independent Gaussian random vectors $(S_0,\ldots ,S_{b-1})$ such that $\text{Cov}(S_j)=\text{Cov}(Y_j)$ for each $j\in \N$ and 
\begin{equation}\label{z}
\mathbb P\left(\max_{0\leq i\leq b-1}\left|\sum_{j=0}^{i}Y_j-S_j\right|\geq Mz\right)\leq C'z^{-q}+\exp(-C'z),
\end{equation}
for all $z\geq C'\log s$. Here, $C'$ is a positive constant which depends only of $C$,  $d$ and  $q$.
\end{proposition}

\begin{lemma}\label{28}
Suppose that $p>2+2/\beta$. Then
for any $n\in\bbN$, there exists a coupling between $(Y_{n,0},\ldots ,Y_{n,F(n)-1})$ and $(S_{n,0},\ldots ,S_{n,F(n)-1})$, where the $S_{n,j}$'s are independent centered Gaussian random variables with $Var(S_{n,j})=Var(Y_{n,j})$, such that
\begin{equation}\label{(5.11)}
\sum_{n} \mathbb P\left(\max_{1\leq i\leq F(n)}\left|\sum_{j=0}^{i-1}Y_{n,j}-S_{n,j}\right|\geq 2^{((1-\beta)/2+(\beta+1)/p+\ve/2)n}\right)<\infty.
\end{equation}
\end{lemma}
\begin{proof}[Proof of Lemma~\ref{28}]
Take $q\in (2,p)$.
By (\ref{Modified bound}), we have that
\begin{equation}\label{a5}
\|Y_{n,j}\|_{L^q}\leq C2^{(1-\beta)n/2+n/p},
\end{equation}
where we have used that the right end point of each $I_{n,j}$ does not exceed $2^{n+1}$ and that $X_{n,j}$ and $Y_{n,j}$ are equally distributed. It follows from~\eqref{a5} that \[M:=\bigg(\sum_{j=0}^{F-1}\|Y_{n,j}\|_{L^q}^q\bigg)^{\frac 1q}\] satisfies
\[
M\leq C2^{n/p+\be n/q+(1-\beta)n/2}. 
\]
Therefore, if $q$ is sufficiently close to $p$  then $M^2$ is much smaller than $2^n$, where we have used that $p>2+2/\beta$.  On the other hand, by (\ref{Go1}) we have 
\begin{equation}\label{V est}
Var(Y_{n,j})=Var(X_{n,j})\geq u 2^{(1-\beta)n}
\end{equation}
for some constant $u>0$ which does not depend on $n$ and $j$. Here we have taken into account that the length of each $I_{n,j}$ is of magnitude $2^{(1-\be)n}$.
By (\ref{V est}) we have
\begin{equation}\label{434x}
Var\bigg (\sum_{j=0}^{F-1}Y_{n,j} \bigg)=\sum_{j=0}^{F-1}Var\big (Y_{n,j} \big)\geq c 2^n,
\end{equation}
where $c>0$ is some  constant.

Next, set $v_j=v_{n,j}=Var(Y_{n,j})$. Then $v_j\leq \|Y_{n,j}\|_{L^q}^2\leq M^2$. Let $u_1$ be the largest index such that 
\[
v_0+\ldots +v_{u_1-1}\geq 100M^2.
\]
Such index exists since $\sum_{j=0}^{F-1} v_j$ is much larger than $M^2$ (see~\eqref{434x}). Notice now that 
\[
v_0+\ldots +v_{u_1-1}\leq v_0+\ldots +v_{u_1-2}+M^2\leq 101M^2.
\]
This gives us the first block $\{Y_{n,0}, \ldots, Y_{n, u_1-1}\}$ of consecutive $Y_{n,j}$'s from the proof of~\cite[Lemma 5.6]{GO} such that~\eqref{BLOCKS} holds. We can continue by  forming $k+1$ consecutive blocks, namely  \[\{Y_{n,0},\ldots ,Y_{n,u_1-1}\},\ldots ,\{Y_{n, u_k},\ldots ,Y_{n,u_{k+1}-1}\},\] where $k$ is the first step in the construction such that \[v_{u_{k+1}}+\ldots +v_{F}<100M^2.\] Then, we add $Y_{n,u_{k+1}},\ldots ,Y_{n,F}$ to the last block  $\{Y_{u_k},\ldots ,Y_{n,u_{k+1}-1}\}$ we have constructed. This means that we can always assume that the sum of the variances of the random variables $Y_j=Y_{n,j}$ along successive blocks is not less than $100M^2$ and that it doesn't  exceed $201M^2$. The statement  of the lemma now follows by applying Proposition~\ref{Cor3} with $z=2^{\ve n/2}$, taking into account that the number of blocks is trivially bounded by $F=F(n)$.
\end{proof}

\emph{Third step:} It follows from the previous two steps of the proof that, when $p>2+2/\beta$ there exists a coupling between $(A_n)_{n\in\cI}$ and a sequence $(B_n)_{n\in\cI}$ of independent centered normal random variables  so that when $(n,j)$ tends to infinity, we have 
\[
\left|\sum_{\ell<i_{n,j},\ell\in\cI}(A_\ell-B_\ell)\right|=o(2^{(\beta+\ve)n/2}+ 2^{((1-\beta)/2+(\beta+1)/p+\ve)n}),
\] 
where $i_{n,j}$ denotes the smallest element of $I_{n, j}$. We note  that we have also used the  so-called Berkes–Philipp lemma (see \cite[Lemma A.1]{BP} or \cite[Lemma 3.1]{GO}).

\emph{Fourth step:}
We now establish the version of~\cite[Lemma 5.8]{GO}. However, before we do that we need the following result, which is a consequence of \cite[Theorem 1]{Morc} (see also~\cite[Corollary B1]{Serfling}).
\begin{lemma}\label{be1}
Let $Y_1,\ldots ,Y_d$ be a finite sequence of random variables. Let $v>2$ be finite and assume that there exist constants $C_1, C_2>0$ such that $\|Y_i\|_{L^v}\leq C_1$  for every $i\in \{1, \ldots, d\}$. Moreover,  assume  that for any $a, n\in \N$ satisfying $a+n\leq d$, we have that 
\[
\|S_{a,n}\|_{L^v}\leq C_2^2n^{\frac12},
\]
 where \[S_{a,n}=\sum_{i=a+1}^{a+n}Y_i.\] Then, there exists a constant $K>0$ (depending only on $C_1,C_2$ and $v$) such that for any $a$ and $n$,
 \begin{equation}\label{Bou}
 \|M_{a,n}\|_{L^v}\leq Kn^{\frac12},
 \end{equation}
 where \[M_{a,n}=\max\{|S_{a,1}|,\ldots,|S_{a,n}|\}.\]
 \end{lemma}

The following is the already announced version of~\cite[Lemma 5.8]{GO}. 
\begin{lemma}\label{Lemma 5.8}
We have  that as $(n,j)\to\infty$,
\begin{equation}\label{4th}
\max_{m<|I_{n,j}|}\left|\sum_{\ell=i_{n,j}}^{i_{n,j}+m}A_\ell \right|=o(2^{((1-\beta)/2+\beta/p+1/p+\ve)n}) \quad \text{a.s.} 
\end{equation}

\end{lemma}

\begin{proof}[Proof of Lemma \ref{Lemma 5.8}]
Let $q\in(2,p)$. Consider the finite sequence \[Y_k=A_{k}/({i_{n,j}+|I_{n,j}|})^{1/p}, \quad  k\in I_{n,j}.\] Then, by (\ref{Go2}) there exists a constant $C_1>0$ which does not depend on $n$ and $j$ so that $\|Y_k\|_{L^q}\leq C_1$, for any $k \in I_{n,j}$. Moreover, by (\ref{Modified bound}), there exists a constant $C_2>0$ which does not depend on $n$ and $j$ so that for any relevant $a$ and $b$,
\[
\left\|\sum_{k=a+1}^{a+b}Y_k\right\|_{L^q}\leq C_2b^{\frac12}.
\]
Using the same notation as in statement of Lemma~\ref{be1}, we observe that it follows from~\eqref{Bou} that 
\[
\|M_{n,b}\|_{L^q}\leq Kb^{\frac12},
\]
for some constant $K>0$ (which depends only $C_1,C_2$ and $q$).

 In particular, by setting  $v=(1-\beta)/2+\beta/p+\ve/2$, we have that 
\[
\mathbb P(M_{i_{n,j},|I_{n,j}|}\geq 2^{vn})\leq \| \ M_{i_{n,j},|I_{n,j}|}\|_{L^q}^q/2^{vnq}\leq 	K|I_{n,j}|^{q/2}/2^{vnq}.
\]
Moreover, observe that 
\[
\sum_{n,j}|I_{n,j}|^{q/2}/2^{vnq}\leq\sum_{n}2^{\beta n}2^{(1-\be)nq/2-vnq}.
\]
Notice  that the above sum is finite  if $q$ is sufficiently close  to $p$. Applying the Borel-Cantelli lemma yields that, as $(n,j)\to\infty$,
\begin{equation*}
\max_{m<|I_{n,j}|}\left|\sum_{\ell=i_{n,j}}^{i_{n,j}+m}Y_\ell \right|=o(2^{((1-\beta)/2+\beta/p+\ve)n}),  
\end{equation*}
which implies that~(\ref{4th}) holds (since the right end point of $I_{n,j}$ does not exceed $2^{n+1}$).
 \end{proof}

\emph{Fifth step:} By combining the last two steps, we derive that when $k$ tends to infinity, 
\[
\left|\sum_{\ell<k,\,\ell\in\cI}(A_\ell-B_\ell)\right|=o(k^{(\beta+\ve)/2}+ k^{(1-\beta)/2+(\beta+1)/p+\ve})
\]
assuming that $p>2+2/\beta$.

\emph{Sixth step:}
Fix some $n$ and consider the finite sequence $Y_i=A_i/n^{1/p}$ where $i\in \{1,\ldots,n\}$. It follows from our assumptions that
$(Y_i)_i$ satisfies property~(H) (with constants that do not depend on $n$). Applying \cite[Lemma 5.9]{GO} with the finite sequence $(Y_i)$ (instead of $A_i$ there), we see that for any $\al>0$, there exists $C=C_\al$ (which does not depend on $n$) such that for any interval $J\subset[1,n]$ we have
\begin{equation}\label{L5.9}
n^{-2/p} \mathbb E\left|\sum_{\ell\in J\cap\cJ}A_i\right|^2=\mathbb E\left|\sum_{\ell\in J\cap\cJ}Y_i\right|^2\leq C|J\cap\cJ|^{1+\alpha}.
\end{equation}
We recall the following version of the Gal-Koksma law of large numbers, which is a direct consequence of~\cite[Theorem 3]{Morc}  together with some routine estimates (as those given in  the proof of~\cite[Theorem 6]{Morc}). We also note that the lemma can be proved by an easy adaptation of the arguments in the proof of \cite[Theorem A1]{PS}.
\begin{lemma}\label{p4}
Let $Y_1,Y_2,\ldots$ be a sequence of random variables such that with some constants $\sig\geq1$, $C>0$, $p>1$ and for any  $m, n\in \N$ we have that 
\[
\left\|\sum_{j=m+1}^{m+n}Y_j\right\|_{L^2}^2\leq C\big((n+m)^{\sigma}-m^{\sig}\big)\cdot (n+m)^{\frac2p}.
\] 
Then, for any $\del>0$ we have that $\mathbb P$-a.s. as $n\to\infty$,
\[
\sum_{j=1}^nY_j=o(n^{\sig/2+1/p}\ln^{3/2+\del}n).
\]
\end{lemma}

Relying on (\ref{L5.9}) and Lemma \ref{p4},
one can now repeat the arguments appearing  after the statement of~\cite[Lemma 5.9]{GO} with the finite sequence $\big(A_i/k^p\big)_{1\leq i\leq k}$ (instead of $(A_i)_i$), and conclude that 
\[
\sum_{\ell<k,\,\ell\in\cJ}A_{\ell}/k^{\frac 1p}=o(k^{\beta/2+\ve}).
\]
\emph{Finalizing the proof:}
Combining the estimates from the previous steps we  get a coupling of $(A_\ell)$ with independent centered normal random variables $(B_\ell)$ such that 
\[
\left|\sum_{\ell<k}(A_k-B_k)\right|=o(k^{\beta/2+\ve+\frac 1p}+k^{(1-\be)/2+(\beta+1)/p+\ve}),\quad \text{a.s.}
\]
Taking $\beta=p/(2p-2)$, we obtain (\ref{8:44}).  Observe that for this choice of $\beta$ we have $p>2+2/\beta$ since $p\geq 6$. When $4< p<6$ we can make a different choice of $\beta$ and obtain a slightly less attractive rate. To complete the proof of Theorem \ref{Gouzel Thm}, it remains to estimate the variance of the approximating Gaussian $G_n=\sum_{j=1}^n B_j$. Firstly, by applying~\cite[Proposition 9]{DH} with the finite sequence $\big{(}A_i/2^{(n+1)/p} \big{)}_{1\leq i\leq 2^{n+1}}$ replacing $(A_i)_i$,  we obtain that 
\[
\Big\|\sum_{(n',j')\prec (n,j)}X_{n',j'}-Y_{n',j'}\Big\|_{L^2}\leq C2^{\be n/2+n/p},
\]
where $(Y_{n',j'})$ are given by  Proposition \ref{bb}. Since $Y_{n',j'}$ and $S_{n',j'}$ have the same variances,  we conclude that 
\begin{equation}\label{Up}
\left|\Big\|\sum_{(n',j')\prec (n,j)}X_{n',j'}\Big\|_{L^2}-\Big\|\sum_{(n',j')\prec (n,j)}S_{n',j'}\Big\|_{L^2}\right|\leq C2^{\be n/2+n/p}.
\end{equation}
Take $n\in\bbN$, and let $N_n$ be such that $2^{N_n}\leq n<2^{N_n+1}$. Furthermore, let   $j_n$ be the largest index such that the left end point of $I_{N_n,j_n}$ is smaller than $n$. 
In the case when $n\in I_{N_n,j_n}$ we have
\[
\begin{split}
\sum_{i=1}^n A_i-\sum_{(n',j')\prec (N_n,j_n)}X_{n',j'} &=\sum_{(n',j')\prec (N_n,j_n)}\,\sum_{i\in J_{n',j'}}A_i+\sum_{i\in J_{N_n,j_n}}A_i\\
&\phantom{=}+\sum_{i=i_{N_n,j_n}}^{n}A_i \\
&=\sum_{i\leq n, i\in J}A_i+\sum_{i=i_{N_n,j_n}}^{n}A_i \\
&=:I_1+I_2.
\end{split}
\]
 Recall next that by \cite[(5.1)]{GO} the cardinality of $\cJ\cap[1,2^{N_n+1}]$ does not exceed $C2^{\ve (N_n+1)}2^{\be N_n}(\ve N_n+2)$, which for our specific choice of $N_n$ is at most $Cn^{\be +3\ve/2}$ (where $C$ denotes a generic constant independent of $n$).
Using (\ref{L5.9}) with a sufficiently small $\al$ we derive that
\[
\|I_1\|_{L^2}\leq C n^{1/p+\be/2+\ve}.
\]
On the other hand, applying (\ref{Modified bound}) we obtain that
\begin{eqnarray*}
\|I_2\|_{L^2}\leq C|I_{N_n,j_n}|^{\frac12}2^{N_n/p}\\\leq C2^{N_n(1-\beta)/2+N_n/p}\leq Cn^{(1-\beta)/2+1/p}\leq Cn^{\be/2+1/p}
\end{eqnarray*}
where we have used that for our specific choice of $\beta$ we have $(1-\be)/2=\be/2-\be/p<\be/2$.
We conclude that there exists a constant $C'>0$ so that for any $n\geq1$,
\[
\left\|\sum_{j=1}^{n}A_j-\sum_{(n',j')\prec (N_n,j_n)}X_{n',j'}\right\|_{L^2}\leq C'n^{\be/2+\ve +1/p}.
\]
The proof of (\ref{Var est}) in the case when $n\in I_{N_n,j_n}$ is completed now using (\ref{Up}). 
The case when $n\not\in I_{N_n,j_n}$ is treated similarly. We first write
\[
\sum_{i=1}^n A_i-\sum_{(n',j')\prec (N_n,j_n)}X_{n',j'}=
\sum_{j\in \cJ, j\leq n}A_i+X_{N_n,j_n}:=I_1+I_2.
\]
Then the $L^2$-norms of $I_1$ and $I_2$ are bounded exactly as in the case when $n\in I_{N_n,j_n}$, and the proof of~\eqref{Var est} is complete. Finally,  the last conclusion in the statement of the theorem follows directly  from~\eqref{8:44}, \eqref{Var est} together with~\cite[Theorem 3.2A]{HR}, \cite[Lemma A.1]{BP} (seel also~\cite[Lemma 3.1]{GO}) and the so-called Strassen-Dudley theorem~\cite[Theorem 6.9]{Bil} (see also~\cite[Theorem 3.4]{GO}).
\end{proof}

\section{Main result}\label{Main}
The goal of this section is to establish the quenched  almost sure invariance principle for random distance expanding maps satisfying suitable conditions. This is done by applying Theorem~\ref{Gouzel Thm}.

Without any loss of generality, we can suppose that our observable $\psi \colon \cE \to \mathbb R$ is fiberwise centered, i.e. that $\int_{\cE_\om }\psi_\om \, d\mu_\om =0$ for $\mathbb P$-a.e. $\om \in \Omega$. Indeed, otherwise we can simply replace $\psi$ with $\tilde \psi$ given by 
\[
\tilde \psi_\om=\psi_\om-\int_{\cE_\om}\psi_\om\, d\mu_\om, \quad \om \in \Omega. 
\]
In what follows, $\mathbb E_\om(\varphi)$ will denote the expectation of   a measurable $\varphi \colon \cE_\om \to \mathbb R$ with respect to $\mu_\omega$. 
The proof of the following result can be obtained by repeating the arguments from~\cite[Lemma 12.]{DFGTV1} and~\cite[Proposition 3.]{DFGTV1} (see also  ~\cite[Theorem 2.3.]{kifer})
\begin{proposition}
We have the following:
\begin{enumerate}
\item there exists $\Sigma^2 \ge 0$ such that 
\begin{equation}\label{827}
\lim_{n\to \infty} \frac 1 n \mathbb E_\om\bigg{(}\sum_{k=0}^{n-1}\psi_{\sigma^k \om}\circ f_\omega^k \bigg{)}^2=\Sigma^2, \quad \text{for $\mathbb P$-a.e. $\om \in \Omega$;}
\end{equation}
\item $\Sigma^2=0$ if and only if there exists $\varphi \in L_\mu^2 (\mathcal E)$ such that 
\[
\psi=\varphi -\varphi \circ F.
\]
\end{enumerate}
\end{proposition}
From now on we shall assume that $\Sigma^2>0$.  For any integer $L\geq 1$ 
consider the set
\[
A_L=\left\{\om\in\Om:\,\frac1n \mathbb E_\om\bigg{(}\sum_{k=0}^{n-1}\psi_{\sigma^k \om}\circ f_\omega^k \bigg{)}^2\geq\frac12\Sigma^2, \,\,\,\,\,\forall n\geq L\right\}.
\]
Then $A_{L}\subset A_{L'}$ if $L\leq L'$ and 
 the union of the $A_L$'s has probability $1$.
Due to measurability of $Q_\om,C(\om)$, $K(\om)$, and $\om \mapsto h_\om$, for any   $C_0 >0$ and $L\in \mathbb N$ the set 
\begin{equation}\label{8:10}
E:=\{ \om \in  \Omega : \max \{C(\om), K(\om), \lVert h_\om\rVert_ {\infty}, \lVert 1/h_\om \rVert_{\al,\xi},1/Q_\om\} \le C_0 \}\cap A_L
\end{equation}
is measurable, and when $C_0$ and $L$ are sufficiently large we have that 
$\mathbb P(E)>0$. Fix some large enough $C_0$ and $L$, and for $\om \in \Omega$, let   
\[
m_1(\om):=\inf \{n\in \mathbb N: \sigma^n \om\in E\}.
\]
For $k>1$ we inductively define
\[
m_k(\om):=\inf \{n >m_{k-1} (\om): \sigma^n \om \in E\}.
\]
Due to ergodicity of $\mathbb P$, we have that 
$m_k(\om)$ is well-defined for $\mathbb P$-a.e. $\om \in \Omega$ and every $k\in \mathbb N$. Let us consider the associated induced system $(E, \mathcal F_E, \mathbb P_E, \iota)$, where
$\mathcal F_E=\{A\cap E: A\in \mathcal F\}$, $\mathbb P_E(A)=\frac{\mathbb P(A)}{\mathbb P(E)}$, $A\in \mathcal F_E$ and $\iota (\om)=\sigma^{m_1(\om)}\om$ for $\om \in E$. We recall that $\mathbb P_E$ is invariant for $\iota$ and in fact ergodic. 

It follows from Birkhoff's ergodic theorem that 
\begin{equation}\label{Kn lim}
\lim_{n\to \infty} \frac{k_n(\om)}{n}=\mathbb P(E) \quad \text{for $\mathbb P$-a.e. $\om \in \Omega$,}
\end{equation}
where
\[
k_n(\om):=\max \{k\in \mathbb N: m_k(\om)\le n\}.
\]
Moreover, Kac's lemma implies that 
\[
\lim_{n\to \infty} \frac{m_n(\om)}{n}=\frac{1}{\mathbb P(E)}, \quad \text{for $\mathbb P$-a.e. $\om \in \Om$.}
\]
By combining the last two equalities, we conclude  that
\[
\lim_{n\to \infty}\frac{m_{k_n(\om)}(\om)}{n}=1, \quad \text{for $\mathbb P$-a.e. $\om \in \Om$.}
\]
For $\mathbb P$ a.e. $\om \in \Om$, set
\[
\Psi_\om:=\sum_{j=0}^{m_1(\om)-1}\psi_{\sigma^j \om}\circ f_\om^j. 
\]
We assume that there exists $p\geq 6$, so that
\begin{equation}\label{10:07}
\text{the map $\om \mapsto A(\om):=\lVert \Psi_\om \rVert_\infty$ belongs to $L^p (\Om, \mathcal F, \mathbb P)$.}
\end{equation}
Finally, let $L_\om:=\cL_\om^{m_1(\om)}$ and $F_\om:=f_\om^{m_1(\om)}$, for $\om \in \Omega$.

We are now in a position to state the main result of our paper (recall our assumption that $\Sig^2>0$).
\begin{theorem}\label{main}
For $\mathbb P$-a.e. $\omega \in \Omega$ and arbitrary $\del>0$, there exists a coupling between  $(\psi_{\sigma^i\omega} \circ f_\omega^i)_{i}$, considered as a sequence of random variables on $(\cE_\om,\mu_\om)$,
and a  sequence $(Z_k)_k$ of independent centered (i.e. of zero mean) Gaussian random variables
such that
 \begin{equation}\label{sca}
\bigg{\lvert} \sum_{i=1}^n  \psi_{\sigma^i\omega} \circ f_\omega^i-\sum_{i=1}^n Z_i \bigg{\rvert}=o(n^{a_p+\del}), \quad \text{a.s.},
\end{equation}   
where \[a_p=\frac{p}{4(p-1)}+\frac 1p. \]
Moreover, there exists $C=C(\om)>0$ so that for any $n\geq1$,
\begin{equation}\label{Var est Application}
\Big\|\sum_{i=1}^n \psi_{\sigma^i\omega} \circ f_\omega^i\Big\|_{L^2}-Cn^{a_p+\delta}\leq \Big\|\sum_{i=1}^n Z_i\Big\|_{L^2}\leq \Big\|\sum_{i=1}^n \psi_{\sigma^i\omega} \circ f_\omega^i\Big\|_{L^2}+Cn^{a_p+\delta}.
\end{equation} 
Finally, there exists a coupling between  $(\psi_{\sigma^i\omega} \circ f_\omega^i)_i$ and a standard Brownian motion $(W_t)_{t\ge 0}$ such that 
\[
\left|\sum_{i=1}^{n}\psi_{\sigma^i\omega} \circ f_\omega^i-W_{\sigma_{\om, n}^2}\right|=o(n^{\frac12 a_p+\frac14+\delta}) \quad a.s., 
\]
where
\[
\sigma_{\om,n}=\Big\|\sum_{i=1}^n\psi_{\sigma^i\omega} \circ f_\omega^i \Big\|_{L^2}.
\]

\end{theorem}
\begin{remark}
Observe that $a_p \to \frac 1 4 $ as $p\to \infty$. We note that our proof also yields convergence rate when $4<p<6$, which has a slightly less attractive form in terms of $p$.  In addition, we emphasize that $\Big\|\sum_{i=1}^n Z_i\Big\|_{L^2}$ depends on $\omega$ but that it is asymptotically deterministic. More precisely, it follows from~\eqref{827} and~\eqref{Var est Application} that 
\[
\lim_{n\to \infty} \frac{\Big\|\sum_{i=1}^n Z_i\Big\|_{L^2}^2}{n\Sigma^2}=1.
\]
\end{remark}

\begin{proof}[Proof of Theorem~\ref{main}]
Our strategy  proceeds as follows. Firstly, we will apply Theorem~\ref{Gouzel Thm} to establish the invariance principle for the induced system. Secondly, we extend the invariance principle to our original system.  Throughout the proof, $C>0$ will denote a generic constant independent on $\om$ and other parameters involved in the estimates. 

For $\omega  \in E$ (recall that $E$ is given by~\eqref{8:10}), set $A_n=\Psi_{\iota^n \om}\circ F_\om^n$, $n\in \mathbb N$. Obviously, $A_n$ depends also on $\om$ but in order to make the notation as simple as possible, we do not make this dependence explicit.

Observe that it follows from~\eqref{10:07} and Birkhoff's ergodic theorem that there exists a random variable $R\colon E \to (0, \infty)$ such that:
\begin{equation}\label{8:23}
\lVert A_n\rVert_{L^p} \le R(\om) n^{1/p} \quad \text{for $\mathbb P$-a.e. $\om \in E$ and $n\in \N$.}
\end{equation}
It follows easily from~\eqref{8:11} and~\eqref{8:10} that for any $k\in \mathbb N$, $n\ge L$ and $\om \in E$, 
\begin{equation}\label{8:24}
\frac 1 n  Var \bigg{(}\sum_{j=0}^{n-1} A_{j+k}\bigg{)} \ge \frac{1}{2}\Sig^2,
\end{equation}
where we have used that $m_n(\iota^k (\om)) \ge n$.
We conclude  from~\eqref{8:23} and~\eqref{8:24} that the processes $(A_n)_{n\in \N}$ satisfies~\eqref{Go2} and~\eqref{Go1}, respectively. 

Hence, in order to apply Theorem~\ref{Gouzel Thm}, we need to show that $(A_n)_{n\in \N}$ satisfies property~(H) and,  in addition,  that for any $n<m$ the finite sequence $(A_i/(n+m)^{1/p})_{n+1\le i \le n+m}$ also satisfies~(H) (with uniform constants). In fact, we will prove the following: the process $(a_nA_n)_{n\in \N}$ satisfies~(H) for any sequence $(a_n)_{n\in \N} \subset (0, 1]$ (and with uniform constants). Let us begin by introducing some
auxiliary notations. For $\mathbb P$ a.e. $\om \in \Omega$ and $z\in \mathbb C$, let
\[
\hat{\cL}_\om^zg:=\hat{\cL}_\om (ge^{z\psi_\om})=\cL_\om (ge^{z\psi_\om}h_\om)/h_{\sigma \om}, \quad  \text{for $g\in \cH_\om^{\al,\xi}$.}
\]
Furthermore, for $z\in \mathbb C$ and $n\in \N$, set
\[
\hat{\cL}_\om^{z, n}:=\hat{\cL}_{\sigma^{n-1}\om}^z \circ \ldots \circ  \hat{\cL}_\om^z.
\]
It is easy to verify that
\[
\hat{\cL}_\om^{z, n} g=       \cL_\om^n (ge^{zS_n^\om \psi}h_\om)/h_{\sigma^n \om}=\cL_\om^{z, n}(gh_\om)/h_{\sigma^n \om}.
\]
Finally, for $\om \in \Omega$,  $n\in \N$ and $\overline{t}=(t_0, t_1, \ldots, t_{n-1})\in \mathbb R^n$, let 
\[
L_\om^{\overline{t}, n}=     \hat{\cL}_{\iota^{n-1} \om}^{it_{n-1}, m_n(\om)-m_{n-1}(\om)}\circ  \ldots \circ             \hat{\cL}_{\iota \om}^{it_1, m_2(\om)-m_1(\om)} \circ   \hat{\cL}_\om^{it_0, m_1(\om)}.
\]
Observe that 
\[
L_\omega^{\overline{t},n}g=(\cL_{\iota^{n-1}\om}^{it_{n-1},m_n(\om)-m_{n-1}(\om)}\circ\ldots \circ\cL_{\iota \om}^{it_1,m_2(\om)-m_1(\om)}\circ\cL_{\om}^{it_0,m_1(\om)})(gh_\om)/h_{\iota^{n}\om},
\]
for any $g\in \cH_\om^{\al,\xi}$. It follows from~\eqref{552n}, \eqref{8:10} and the above formula that  for $n\in \N$ and 
$\overline t \in [-1, 1]^n$, we have that
\begin{equation}\label{8:37}
\lVert L_\omega^{\overline{t},n} \rVert_{\alpha,\xi} \le C.
\end{equation}
For $\om \in \Omega$ and $g\in  \cH_\om^{\al,\xi}$, set
\[
\Pi_\om g:=\bigg{(} \int_{\cE_\om}g\, d\mu_\om \bigg{)}\mathbf{1}
\]
where $\textbf{1}$ denotes the function which takes the constant value $1$, regardless of the space on which it is defined.
Since $L_\omega^{\overline{0},k}=\hat{\mathcal L}_\om^{m_k(\om)}$ and $m_k(\om)\ge k$, it follows from Lemma~\ref{ds} and~\eqref{8:10}  that 
\begin{equation}\label{5:59}
\lVert ( L_\omega^{\overline{0},k}-\Pi_\om )g \rVert_\infty  \le Ce^{-\lambda k}\lVert g\rVert_{\al,\xi}, 
\end{equation}
for $\om \in E$, $g\in \cH_\om^{\al,\xi}$ and $k\in \N$. 

Take now $n,m, k\in \N$, $b_1<b_2< \ldots <b_{n+m+k}$  and $t_1,\ldots ,t_{n+m}\in\bbR$ with $|t_j|\leq 1$. We have that 
\[
\begin{split}
& \bbE_{\mu_\om}\big(e^{i\sum_{j=1}^nt_j(\sum_{\ell=b_j}^{b_{j+1}-1}B_\ell)+i\sum_{j=n+1}^{n+m}t_j(\sum_{\ell=b_j+k}^{b_{j+1}+k-1}B_\ell)}\big)\\
&=\bbE_{\mu_{\iota^{b_{n+m+1}+k}\om)}}\big(  L_{\iota^{b_{n+1}+k}\om}^{\overline{t}, b_{n+m+1}-b_{n+1}}  L_{\iota^{b_{n+1}}\om}^{\overline{0}, k}  L_{\iota^{b_1}\om}^{\overline{s}, b_{n+1}-b_1} \mathbf 1 \big),\\
\end{split}
\]
where $B_n=a_nA_n$,
\[
\overline{s}=(a_{b_1}t_1, \ldots, a_{b_2-1}t_1, a_{b_2}t_2, \ldots, a_{b_3-1}t_2, \ldots, a_{b_n}t_n, \ldots, a_{b_{n+1}-1}t_n), 
\]
and
\[
\overline{t}=(a_{b_{n+1}+k}t_{n+1}, \ldots, a_{b_{n+2}+k-1}t_{n+1}, \ldots, a_{b_{n+m}+k}t_{n+m}, \ldots, a_{b_{n+m+1}+k-1}t_{n+m}).
\]
Consequently, 
\[
\begin{split}
& \bbE_{\mu_\om}\big(e^{i\sum_{j=1}^nt_j(\sum_{\ell=b_j}^{b_{j+1}-1}B_\ell)+i\sum_{j=n+1}^{n+m}t_j(\sum_{\ell=b_j+k}^{b_{j+1}+k-1}B_\ell)}\big)\\
&=\bbE_{\mu_{\iota^{b_{n+m+1}+k}\om}}(  L_{\iota^{b_{n+1}+k}\om}^{\overline{t}, b_{n+m+1}-b_{n+1}} \big(  L_{\iota^{b_{n+1}}\om}^{\overline{0}, k} -\Pi_{\iota^{b_{n+1}}\om}) L_{\iota^{b_1}\om}^{\overline{s}, b_{n+1}-b_1} \mathbf 1 \big) \\
&\phantom{=}+\bbE_{\mu_{\iota^{b_{n+m+1}+k}\om}}\big(  L_{\iota^{b_{n+1}+k}\om}^{\overline{t}, b_{n+m+1}-b_{n+1}} \Pi_{\iota^{b_{n+1}}\om} L_{\iota^{b_1}\om}^{\overline{s}, b_{n+1}-b_1} \mathbf 1 \big) \\
&=:I_1+I_2.
\end{split}
\]
We claim next that 
\begin{equation}\label{I1 claim}
\lvert I_1\rvert \le Ce^{-\lambda k}.
\end{equation}
Indeed, set 
\[
A:=L_{\iota^{b_{n+1}+k}\om}^{\overline{t}, b_{n+m+1}-b_{n+1}},\,\,
B:=L_{\iota^{b_{n+1}}\om}^{\overline{0}, k} -\Pi_{\iota^{b_{n+1}}\om}
\,\,\text{ and }\,\,g:=L_{\iota^{b_1}\om}^{\overline{s}, b_{n+1}-b_1} \mathbf 1.
\]
Then, 
\[
\|A\|_\infty:=\sup_{f:\|f\|_\infty=1}\|Af\|_\infty\leq \|L_{\iota^{b_{n+1}+k}\om}^{\overline{0}, b_{n+m+1}-b_{n+1}}\textbf{1}\|_\infty=\|\textbf 1\|_\infty=1,
\]
and therefore
\[
|I_1|\leq \|A(Bg)\|_\infty\leq\|A\|_\infty \cdot \|Bg\|_\infty\leq\|Bg\|_\infty.
\]
Applying~\eqref{8:37} we have
\[
\|g\|_{\al,\xi}\leq C,
\]
and thus it follows from~\eqref{5:59}  that 
\[
|I_1|\leq\|Bg\|_\infty\leq Ce^{-\la k}.
\]
We conclude that  (\ref{I1 claim})  holds.

 On the other hand,
\[
I_2=\mathbb E_\om  \big{(}e^{i\sum_{j=1}^nt_j(\sum_{\ell=b_j}^{b_{j+1}-1}B_\ell)}\big{)} \cdot \mathbb E_\om \big{(}e^{i\sum_{j=n+1}^{n+m}t_j(\sum_{\ell=b_j+k}^{b_{j+1}+k-1}B_\ell)}\big{)}.
\]
We conclude that the process $(B_n)_{n\in \N}$ satisfies property~(H) with constants that do not depend on the sequence $(a_l)$.   Thus, Theorem~\ref{Gouzel Thm} yields the almost sure invariance principle for the process $(\Psi_{\iota^n \om}\circ F_\om^n)_{n\in \N}$. 

It remains to observe that the conclusion of the Theorem \ref{main} now follows from the Berkes–Philipp lemma (see \cite[Lemma A.1]{BP} or \cite[Lemma 3.1]{GO}) and the
following lemma which  together with (\ref{Kn lim}), ensures that (\ref{Var est Application}) holds true.
\begin{lemma}\label{ae}
There exists a random variable $U\colon \Omega \to (0, \infty)$ such that
\[
\bigg{\lVert}\sum_{j=0}^{n-1}\psi_{\sigma^j \om}\circ f_\om^j -\sum_{j=0}^{k_n(\om)-1} \Psi_{\iota^j \om} \circ F_\om^j \bigg{\rVert}_\infty \le U(\om)n^{1/p},
\]
for $\mathbb P$-a.e. $\om \in \Omega$ and $n\in \N$.
\end{lemma}

\begin{proof}[Proof of the lemma]
If $n=m_{k_n(\om)}(\om)$ then there is nothing to prove, and so we assume that $m_{k_n(\om)}(\om)<n$.
Observe that
\begin{eqnarray*}
\sum_{j=0}^{n-1}\psi_{\sigma^j \om}\circ f_\om^j -\sum_{j=0}^{k_n(\om)-1} \Psi_{\iota^j \om} \circ F_\om^j=\sum_{j=m_{k_n(\om)}(\om)}^{n-1}\psi_{\sigma^j \om}\circ f_\om^j\\=
\sum_{j=m_{k_n(\om)}(\om)}^{m_{k_n(\om)+1}(\om)-1}\psi_{\sigma^j \om}\circ f_\om^j
 -\sum_{j=n}^{m_{k_n(\om)+1}(\om)-1}\psi_{\sigma^j \om}\circ f_\om^j\\=\Psi_{\sigma^{k_n(\om)}\om}-\Psi_{\sigma^n\om}
\end{eqnarray*}
and thus 
\[
\bigg{\lVert}\sum_{j=0}^{n-1}\psi_{\sigma^j \om}\circ f_\om^j -\sum_{j=0}^{k_n(\om)-1} \Psi_{\iota^j \om} \circ F_\om^j \bigg{\rVert}_\infty \le \lVert \Psi_{\sigma^{k_n(\om)}\om}\rVert_\infty+ \lVert\Psi_{\sigma^{n}\om}\rVert_\infty,
\]
where we have used that $\sigma^j\om\notin E$ when $m_{k_n(\om)}(\om)<j<m_{k_n(\om)+1}(\om)$.
Hence, the conclusion of the lemma follows directly from Birkhoff's ergodic theorem, (\ref{Kn lim}) and~\eqref{10:07}.
\end{proof}

\end{proof}

\section{Acknowledgements}
We would like to thank the anonymous referee for his/hers constructive and illuminating comments that helped us to improve our paper.   
D. D. was supported in part by Croatian
Science Foundation under the project IP-2019-04-1239 and by the University of
Rijeka under the projects uniri-prirod-18-9 and uniri-prprirod-19-16.


\begin{thebibliography}{99.}


\bibitem{ANV} R. Aimino, M. Nicol, S. Vaienti, {\em Annealed and quenched limit theorems for random expanding dynamical systems}, Probab. Theory Related Fields  {\bf 162} (2015),  233-274.

\bibitem{Bil}
P. Billingsley, {\em Convergence of Probability Measures}, (1999), 2nd ed. Wiley, New York.

\bibitem{BP}
J. Berkes and W. Philipp, {\em Approximation theorems for independent and weakly
dependent random vectors}, Ann. Probab. \textbf{7}  (1979), 29--54.

\bibitem{CV}
C. Castaing and M.Valadier, \emph{Convex analysis and measurable multifunctions}, 
Lecture Notes Math., vol. 580, Springer, New York, 1977.


\bibitem{CM} C. Cuny and F. Merlev\`{e}de, \emph{Strong invariance principles with rate for "reverse" martingales and applications}, J. Theoret. Probab. \textbf{28} (2015), 137--183.



\bibitem{DFGTV1} D. Dragi\v cevi\' c, G. Froyland, C. Gonzalez-Tokman and S. Vaienti, \emph{Almost Sure Invariance Principle for random piecewise expanding maps}, Nonlinearity \textbf{31} (2018), 2252--2280.

\bibitem{DH}
D. Dragi\v cevi\' c, Y. Hafouta,
\newblock{A vector-valued almost sure invariance principle for random hyperbolic and piecewise-expanding maps},
\newblock Preprint. https://arxiv.org/abs/1912.12332.




\bibitem{FieldMelbourneTorok}
M. Field, I. Melbourne and A. T\"or\"ok,
 \emph{Decay of correlations, central limit theorems and approximation by
   Brownian motion for compact Lie group extensions},
   Ergodic Theory Dynam. Systems,
   {\bf 23}   (2003), 87--110.
   
   


\bibitem{GO} S. Gou\"ezel, \emph{Almost sure invariance principle for dynamical systems by spectral methods},  Ann. Probab.  \textbf{38}  (2010), 1639--1671.




\bibitem{HK}
Y. Hafouta and Y. Kifer, {\em Nonconventional limit theorems and random dynamcis}, World Scientific, 2018, New-Jersey.

\bibitem{HR} D. L. Hanson and R. P. Russo, \emph{Some Results on Increments of the Wiener Process with Applications to Lag Sums of I.I.D. Random Variables}, Ann. Probab. \textbf{11} (1983), 609--623.


\bibitem{HNTV} N. Haydn, M. Nicol, A. T\"or\"ok and S. Vaienti, \emph{Almost sure invariance principle for sequential and non-stationary dynamical systems}, Trans. Amer Math. Soc. \textbf{369} (2017), 5293--5316.





\bibitem{kifer} Y. Kifer, \emph{Limit theorems for random transformations and processes in random environments}, Trans. Amer. Math. Soc. \textbf{350} (1998), 1481--1518.


\bibitem{KO2} A. Korepanov, {\em Equidistribution for nonuniformly expanding dynamical systems},
Comm. Math. Phys.  \textbf{359} (2018), 1123--1138.


\bibitem{Korepanov}
A. Korepanov,
\emph{Rates in Almost Sure Invariance Principle for Dynamical Systems with Some Hyperbolicity},
Comm. Math. Phys. \textbf{363} (2018), 173--190.
  
  



\bibitem{MN1} I. Melbourne and M. Nicol, {\em  A vector-valued almost sure invariance principle for hyperbolic dynamical systems}, Ann. Probab. {\bf 37}  (2009), 478-505.


    \bibitem{MN2} I. Melbourne and M. Nicol, {\em Almost sure invariance principle for nonuniformly hyperbolic systems}, Commun. Math. Phys. {\bf 260} (2005),  131-146.
    
    \bibitem{Morc}
F. M\'orciz, {\em Moment inequalities and the strong laws of large numbers}. Z. Wahrsch.
Verw. Gebiete \textbf{35}  (1976), 299-314.

\bibitem{PS} W. Philip and  W.F. Stout, \emph{Almost sure invariance principle for sums of weakly dependent random variables},  Memoirs of Amer. Math. Soc. \textbf{161}, Amer. Math. Soc., Providence, (1975).




\bibitem{Serfling}
R. J. Serfling, \emph{Moment inequalities for the maximum cumulative sum}. Ann. Math.
Statist. \textbf{41}  (1970), 1227-1234.







\bibitem{MSU}
V. Mayer, B. Skorulski and M. Urba\'nski, \emph{Distance expanding random mappings,
thermodynamical formalism, Gibbs measures and fractal geometry},
Lecture Notes in Mathematics, vol. 2036 (2011), Springer.



\bibitem{Zai}
A.Y. Zaitsev, \emph{Estimates for the rate of strong approximation in the multidimensional
invariance principle}. J. Math. Sci. (2007) 4856--4865.









\end{thebibliography}
\end{document}